\documentclass[letterpaper, 10 pt, conference]{ieeeconf} 
\pdfoutput=1 
\usepackage[english]{babel}
\usepackage{cite}
\usepackage{amsmath,amssymb,amsfonts}
\usepackage{algorithm}
\usepackage{algorithmic}
\usepackage{graphicx}
\usepackage{textcomp}
\usepackage{amsmath}
\usepackage{dsfont}
\usepackage{xcolor}
\usepackage{units}
\usepackage{booktabs}
\usepackage{comment}
\usepackage{url}
\usepackage[normalem]{ulem}
\usepackage[hidelinks]{hyperref}

\usepackage{amsthm}
\newcounter{thm}
\newtheorem{prob}[thm]{Problem}
\newtheorem{lemma}{Lemma}[section]
\newtheorem{theorem}{Theorem}[section]
\newtheorem{definition}{Definition}[section]
\newtheorem{remark}{Remark}[section]
\newtheorem{approximation}{Approximation}[section]
\def\BibTeX{{\rm B\kern-.05em{\sc i\kern-.025em b}\kern-.08em
    T\kern-.1667em\lower.7ex\hbox{E}\kern-.125emX}}
\newcommand{\flexbrac}[1]{\if\relax\detokenize{#1}\relax \else (#1) \fi}
\newcommand{\flexcomma}[1]{\if\relax\detokenize{#1}\relax \else ,#1 \fi}

\newcommand{\abs}[1]{\lvert #1 \rvert}

\newcommand{\cA}{\mathcal{A}}

\newcommand{\cG}{\mathcal{G}}

\newcommand{\cM}{\mathcal{M}}

\newcommand{\cR}{\mathcal{R}}

\newcommand{\cV}{\mathcal{V}}

\begin{document}

\title{\LARGE \bf A Time-invariant Network Flow Model\\ for Two-person Ride-pooling Mobility-on-Demand
}

\author{Fabio Paparella, Leonardo Pedroso, Theo Hofman, Mauro Salazar 
\thanks{Mechanical Engineering, Control System Techonology, Eindhoven University of Technology, PO Box 513 5600 MB Eindhoven, The Netherlands
        {\tt\small \{f.paparella,l.pedroso,t.hofman,m.r.u.salazar\} @tue.nl}}%
}


\maketitle
\thispagestyle{empty}
\pagestyle{empty}

\begin{abstract}
This paper presents a time-invariant network flow model capturing two-person ride-pooling that can be integrated within design and planning frameworks for Mobility-on-Demand systems.
In these type of models, the arrival process of travel requests is described by a Poisson process, meaning that there is only statistical insight into request times, including the probability that two requests may be pooled together.
Taking advantage of this feature, we devise a method to capture ride-pooling from a stochastic mesoscopic perspective.
This way, we are able to transform the original set of requests into an equivalent set including pooled ones which can be integrated within standard network flow problems, which in turn can be efficiently solved with off-the-shelf LP solvers for a given ride-pooling request assignment.
Thereby, to compute such an assignment, we devise a polynomial-time algorithm that is optimal w.r.t.\ an approximated version of the problem.
Finally, we perform a case study of Sioux Falls, South Dakota, USA, where we quantify the effects that waiting time and experienced delay have on the vehicle-hours traveled.
Our results suggest that the higher the demands per unit time, the lower the waiting time and delay experienced by users. In addition, for a sufficiently large number of demands per unit time, with a maximum waiting time and experienced delay of 5 minutes, more than 90\% of the requests can be pooled. 
\end{abstract}

\section{Introduction}
Ride-sharing is a service that is revolutionizing urban transportation.
Within this service, ride-pooling is the concept of having multiple users traveling at the same time on a single vehicle at lower costs, e.g., emissions, energy consumption, fleet size, and also the cost of the ride charged to the user. Nevertheless, these improvements come at the expense of additional waiting time and delays caused by detours.
Ride-pooling is a difficult problem to deal with due to its combinatorial nature. However, sometimes it is enough to study a ride-sharing system from a macroscopic point of view, especially when dealing with mobility planning or design~\cite{SalazarLanzettiEtAl2019,ZardiniLanzettiEtAl2020b,LukeSalazarEtAl2021}. For this reason, the microscopic nature of ride-pooling  is, at first sight, incompatible with an approach on a different scale. In this paper, we propose a framework to deal with ride-pooling from a mesoscopic point of view by moving from a deterministic to a stochastic approach. In particular, we devise a framework to easily incorporate ride-pooling into a linear time-invariant multi-commodity network flow model, also known as traffic flow model, that is a mesoscopic modeling framework usually used for mobility planning and design. 

\textit{Related Literature:} 
This paper pertains to the research streams of traffic flow models and ride-pooling, that we review in the following.
One of the approaches to characterize and control ride-sharing systems is the multi-commodity network flow model, that is suited for easy implementation of many constraints of different nature and can be efficiently solved with commercial solvers.  This model has been used for multiple purposes, from minimizing electricity costs~\cite{RossiIglesiasEtAl2018b,BoewingSchifferEtAl2020} to joint optimization with public transport~\cite{SalazarLanzettiEtAl2019,Wollenstein-BetechSalazarEtAl2021} and the power grid~\cite{RossiZhangEtAl2017,SpieserTreleavenEtAl2014,IglesiasRossiEtAl2018}. For example, Luke et al.~\cite{LukeSalazarEtAl2021} proposed a joint optimization framework for the siting and sizing of the charging infrastructure for an electric ride-sharing system, while in~\cite{PaparellaChauhanEtAl2023} we proposed a simplified, yet more tractable version of the same problem. Yet in all these models the assumption of one person per vehicle is made.

Ride-pooling has been extensively studied. 
Alonso-Mora et al.~\cite{Alonso_Mora_2017} conceived the vehicle group assignment algorithm, which optimally solves the ride-sharing problem with high capacity vehicles in a microscopic setting.
In~\cite{SantiRestaEtAl2013,JintaoHaiEtAl2020} the benefits of vehicle pooling and the pricing and equilibrium in on-demand ride-pooling markets were analyzed, respectively.
Fieldbaum et al.~\cite{FielbaumBaiEtAl2021} studied ride-pooling considering that users can be picked-up and dropped-off within a walkable distance, while in~\cite{FielbaumKucharskiEtAl2022} they examine how to split costs between users that share the same ride. In~\cite{TsaoMilojevicEtAl2019} a time-expanded network flow model is leveraged to compute the optimal routes of a mobility system that allows for ride-pooling. However, in all of these papers the ride-pooling problem has been studied from a microscopic perspective, whereby each request is considered individually.

To the best of the authors' knowledge, a mesoscopic time invariant network flow model that accounts for ride-pooling has not yet been proposed.

\textit{Statement of Contributions:} The main contributions of this paper are threefold. First, we propose a framework to capture ride-pooling, a microscopic combinatorial phenomenon, in a time-invariant network flow model, whereby the arrival process in stochastic. {Second, within the proposed framework, we devise a method to compute a ride-pooling request assignment that is optimal w.r.t. a relaxed version of the minimum fleet size problem.} Third, we showcase our framework with a case study of Sioux Falls, South Dakota, USA, where we show that the inclusion of ride-pooling can significantly benefit such ride-sharing mobility systems.

\textit{Organization:} The remainder of this paper is structured as follows: Section \ref{sec:model} introduces the multi-commodity traffic flow problem and the framework to capture ride-pooling. Section \ref{sec:res} details the case study of Sioux Falls.
Last, in Section \ref{sec:conc}, we draw the conclusions from our findings and provide an outlook on future research endeavors.

\emph{Notation:} Throughout this paper, we denote the vector of ones, of appropriate dimensions, by $\mathds{1}$. The $i$th component of a vector $v$ is denoted by $v_i$ and the entry $(i,j)$ of a matrix $A$ is denoted by $A_{ij}$. The cardinality of set $\mathcal{S}$ is denoted by $\abs{\mathcal{S}}$.
\section{Ride-pooling Network Flow Model}\label{sec:model}
In this section, we first introduce the standard network   traffic flow model~\cite{PavoneSmithEtAl2012}. Then, we extend the formulation to take into account ride-pooling, and finally present a brief discussion on the model.
\subsection{Time-invariant Network Flow Model} 
We model the mobility system as a multi-commodity network flow model, similar to the approaches of~\cite{SalazarLanzettiEtAl2019,LukeSalazarEtAl2021,PaparellaChauhanEtAl2023,RossiIglesiasEtAl2018b,PaparellaSripanhaEtAl2023}. The transportation network is a directed graph $\mathcal{G} = (\mathcal{V}, \mathcal{A})$. {It consists of a set of vertices $\mathcal{V} := \{1,2,...,\abs{\mathcal{V}}\}$,} representing the location of intersections on the road network, and {a set of arcs $\mathcal{A} \subseteq \mathcal{V} \times \mathcal{V}$}, representing the road links between {intersections}. We indicate ${B \in \{-1,0,1\}^{\abs{\cV} \times \abs{\cA}}}$ as the incidence matrix~\cite{Bullo2018} of the road network $\mathcal{G}$. Consider an arbitrary arc indexing of natural numbers $\{1,\ldots,\abs{\cA}\}$, then $B_{ip} = -1$ if the arc indexed by $p$ is directed towards vertex $i$, $B_{ip} = 1$ if the arc indexed by $p$ leaves vertex $i$,  and $B_{ip}=0$ otherwise. We denote $t$ as the vector whose entries are the travel time $t_{a}$ required to traverse each arc $a\in \cA$, ordered in accordance with the arc ordering of $B$, which we assume to be constant. Similarly to \cite{PavoneSmithEtAl2012}, we define travel requests as follows:
\begin{definition}[Requests]
	A travel request is defined as the tuple $r = (o,d,\alpha) \in \mathcal{V} \times \mathcal{V} \times \mathbb{R}_{>0}$, in which $\alpha$ is the number of users traveling from the origin $o$ to the destination $d \neq o$ per unit time.  Define the set of requests as $\mathcal{R} := \{r_m\}_{m\in \mathcal{M}}$, where $\mathcal{M} = \{1,\ldots,M\}$.
\end{definition} 

We assume, without any loss of generality, that the origin-destination pairs of the requests $r_m \in \mathcal{R}$ are distinct. 
In this paper, we distinguish between active vehicle flows, which correspond to the flows of vehicles serving users whether they are ride-pooling or not, and rebalancing flows which correspond to the flows of empty vehicles between the drop-off and pick-up vertices of consecutive requests. We define the active vehicle flow induced by all the requests that share the same origin $i \in {\mathcal{V}}$ as vector $x^{i}$, where element $x^{i}_{a}$ is the flow on arc $a \in \cA$, ordered in accordance with the arc ordering of $B$. The overall active vehicle flow is a {matrix $X \in \mathbb{R}^{\abs{\cA} \times \abs{\cV}}$ defined as $X := \left[x^{1}\  x^2 \, \dots \,x^{\abs{\cV}}\right]$}. The rebalancing flow across the arcs is denoted by $x^{\mathrm{r}} \in \mathbb{R}^{\abs{\cA}}$. In the following, we define the network flow problem.
\begin{prob}[Multi-commodity Network Flow Problem]\label{prob:main}
	Given a road graph $\cG$ and a demand matrix {$D$}, the active vehicle flows $X$ and rebalancing flow $x^\mathrm{r}$ that minimize the cost in terms of overall travel time result from
	\begin{equation*}
		\begin{aligned}
			\min_{X}\; &J(X) =t^\top ( X \mathds{1} + x^\mathrm{r} )  \\
			\mathrm{s.t. }\; & BX = D \\
			&B ( X \mathds{1}+ x^\mathrm{r} )=0 \\
			& X, x^\mathrm{r} \geq 0,
		\end{aligned}
	\end{equation*}
	where the demand matrix $D \in \mathbb{R}^{\abs{\mathcal{V}} \times \abs{\mathcal{V}}}$ represents the requests between every pair of vertices, whose entries are
	\begin{equation}\label{eq:def_D}
		\!\!\!\!D_{ij} = \begin{cases}
			\alpha_m, &  \exists m \in \mathcal{M} : o_m = j \land d_m = i\\
			-\sum_{k\neq j} D_{kj}, & i  = j \\ 
			0, &   \mathrm{otherwise}.
		\end{cases}\!\!\!
	\end{equation}
\end{prob}
Since Problem~\ref{prob:main} is totally unimodular, $X$ and $x^\mathrm{r}$ can be decoupled and computed separately~\cite{Rossi2018}. The objective function can also be interpreted as the minimum fleet size required to implement the flows~\cite{PavoneSmithEtAl2012}. 


\subsection{Ride-pooling Time-invariant Network Flow Model}

In this paper, we propose a formulation to take into account ride-pooling without the need to change the original structure of the problem. We transform the original set of requests, portrayed by $D$, into an equivalent set of requests accounting for ride-pooling, portrayed by $D^\mathrm{rp}$. We define the ride-pooling network flow problem as follows:
\begin{prob}[Ride-pooling Network Flow Problem]\label{prob:rides}
	Given a road graph $\cG$ and a demand matrix $D^\mathrm{rp}$, the active vehicle flows $X$ and rebalancing flow $x^\mathrm{r}$ that minimize the cost in terms of overall travel time result from
	\begin{equation*}
		\begin{aligned}
			\min_{X}\; &J(X) = {t^\top (X \mathds{1} + x^r)} \\
			\mathrm{s.t. }\; & BX = D^\mathrm{rp} \\
			&B ( X \mathds{1}+ x^\mathrm{r} )= 0 \\
			& X, x^\mathrm{r} \geq 0.
		\end{aligned}
	\end{equation*}
\end{prob}

The ride-pooling demand matrix $ D^\mathrm{rp}$ in Problem~\ref{prob:rides}, which describes the pooling pattern, has to be determined according to four key conditions. First, the individual requests, described by $D$, must be served. Second, ride-pooling two requests is only spatially feasible if the detour travel time is not greater than a threshold $\bar{\delta} \in \mathbb{R}_{>0}$. Third, ride-pooling two requests is only temporally feasible if the waiting time for a request to start being served does not exceed a threshold $\bar{t} \in \mathbb{R}_{>0}$. Fourth, the requests are pooled to minimize the cost function of Problem~\ref{prob:rides} at its solution. Due to the combinatorial nature of such an endeavor, we relax the problem in order to attain a computationally tractable algorithm, according to the following approximation.


\begin{approximation}\label{approx}
For the purpose of computing the demand matrix $D^\mathrm{rp}$,  the cost function of Problem~\ref{prob:rides} is approximated by $\tilde{J}(X):=t^\top X \mathds{1}$.
\end{approximation}
This approximation makes sense in the context of the problem, since the active vehicle travel time $t^\top X \mathds{1}$ is usually dominant over the rebalancing travel time $t^\top x^\mathrm{r}$. In fact, in the simulations performed in Section~\ref{sec:res}, $t^\top x^\mathrm{r}$ accounts for less than $3\%$ of $J$ for every scenario studied. Crucially, leveraging Approximation~\ref{approx}, we can devise a polynomial-time algorithm to compute $D^\mathrm{rp}$ that is optimal w.r.t. the approximated version of the problem.


\subsection{Approximate  Computation of the Demand Matrix}

In this section, we present a framework to compute the demand matrix $D^\mathrm{rp}$ under Approximation~\ref{approx}.  


\subsubsection{Spatial Analysis of Ride-pooling}\label{sec:SD}
In this section, we analyze the feasibility and optimal configuration of ride-pooling two requests from a spatial perspective. First, we define $\delta$ as the delay experienced by each user, representing the time required to travel the additional detour distance w.r.t. the scenario without ride-pooling. If the delay experienced by any of the two users is higher than the threshold $\bar{\delta}$, then pooling the two requests is unfeasible. Second, 
given the feasible pooling itineraries, we analyze which one is the optimal, i.e., the best itinerary to serve the requests, and whether pooling is advantageous w.r.t. no pooling.
Consider two requests $r_m,r_n\in \cR$. To restrict this analysis to the spatial dimension, we temporarily make two key considerations, that we lift in Section~\ref{sec:STF}: i)~the requests $r_m$ and $r_n$ are made at the same time; and ii)~both requests have the same demand, which we set, without any loss of generality, to $\alpha=1$. 
\begin{figure}[t]
	\centering
	\includegraphics[width = 0.65\linewidth]{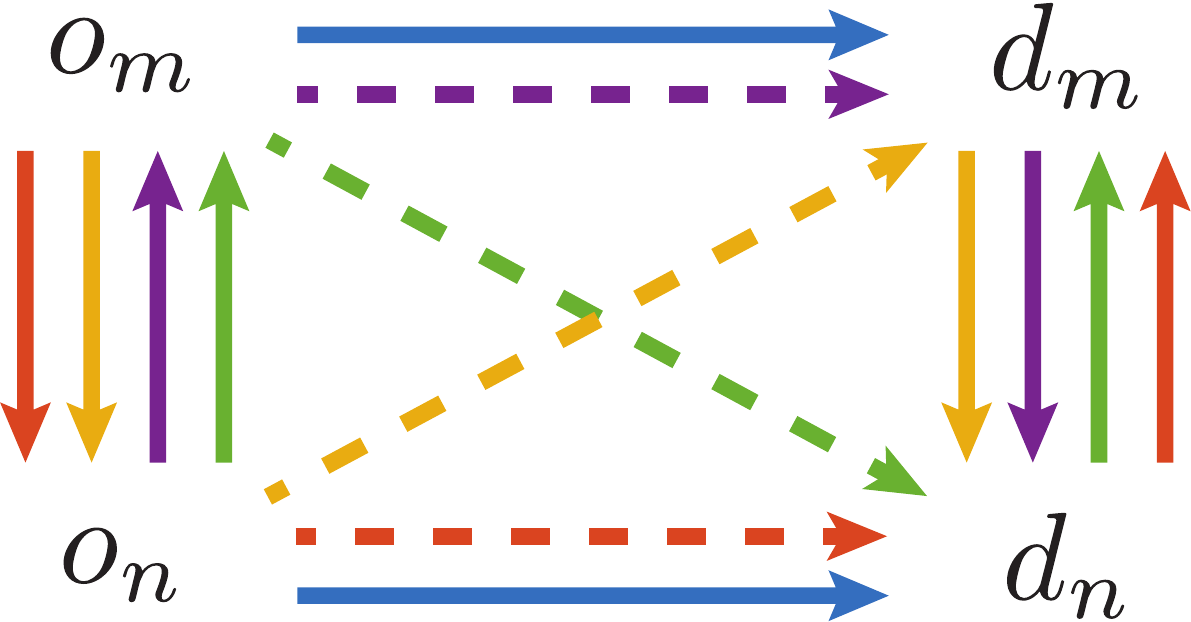}
	\caption{Distinct configurations for serving two requests $r_m,r_n \in \cR$. Each arrow represents a flow of $\alpha = 1$ vehicles. The dashed arrows represent a flow with two users, whilst the solid ones represent a flow with one user.}
	\label{fig:serve_conf}
\end{figure}

There are five different ways of serving two requests $r_m, r_n \in \cR$, as depicted in Fig.~\ref{fig:serve_conf}.
The goal is to assess whether it is feasible to ride-pool $r_m$ and $r_n$ and which is the best configurations among the five. Index each configuration with number $c \in \{0,\ldots,4\}$, with $c=0$ corresponding to no pooling. Each configuration can be split into either two or three equivalent travel requests, as shown in Fig.~\ref{fig:serve_conf}, each corresponding to an arrow. Denote the set of such equivalent requests for configuration $c$ as $\cR_{mn}^c$ ($\cR_{nm}^c = \cR_{mn}^c$) and we define $\Pi(\cR_{mn}^c)$ as the order of visited nodes. For each configuration $c$, one can now solve Problem~\ref{prob:rides}, under Approximation~\ref{approx},  with a simplified demand matrix $D^\mathrm{rp} = D^{mn,c}$ which is obtained from the set of requests $\cR_{mn}^c$ with \eqref{eq:def_D}, obtaining a flow $X^{mn,c} \in \mathbb{R}^{\abs{\mathcal{V}}\times \abs{\mathcal{V}}}$. 
The delay $\delta^{m,c}$ of request $r_m$ for a configuration $c$, is
\begin{equation*}
	\delta^{m,c} = \sum_{p \in {\pi^c_{mn}} } [t^\top X^{mn,c}]_p   - [t^\top X^{mn,0}]_{o_m},
\end{equation*}
where ${\pi^c_{mn}} \subseteq \Pi(\cR_{mn}^c)$ is the ordered set of nodes $\Pi(\cR_{mn}^c)$ from $o_m$ to the node before $d_m$. 
The feasible configurations are those whose delay of both users is below the threshold $\bar{\delta}$. Then, among the feasible ones, comprehending also the no pooling option, the optimal configuration is the one whose flow $X^{mn,c}$ achieves the lowest cost $\tilde{J}(X^{mn,\star})$. Henceforth, the simplified demand matrix of the optimal configuration for ride-pooling $r_m$ and $r_n$ is denoted by $D^{mn,\star}$.

\begin{remark} The demand matrix $D^{mn,c}$ contains either two or three equivalent travel requests. To reduce the computational load, Problem~\ref{prob:rides} can be computed for each equivalent request and stored separately. It amounts to solve Problem~\ref{prob:rides} $\abs{\cV}^2$ times. Since each has a computational complexity of $\mathcal{O} (\abs{\cV}^2)$, the overall computational complexity is $\mathcal{O} (\abs{\cV}^4)$. The procedure depends on the graph $\mathcal{G}$, meaning that the computations have to be performed only once.
\end{remark}
\begin{remark}
The procedure can be extended to account for the possibility of pooling three or more requests. In this case, the number of possible configurations increases, but the overall complexity remains $\mathcal{O} (\abs{\cV}^4)$.
\end{remark}


\subsubsection{Temporal Analysis of Ride-pooling}\label{sec:TD}
In this section, we analyze the temporal alignment of two requests for ride-pooling.  
We derive the probability of two requests taking place within the maximum waiting time, $\bar{t}$. As common in traffic flow models~\cite{PavoneSmithEtAl2012}, we consider that the arrival rate of a request $r_m \in \cR$ follows a Poisson process with parameter $\alpha_m$. Consider two requests $r_m, r_n \in \cR$. In the following lemma, we indicate the probability of 
the two events occurring within a maximum time window $\bar{t}$.
%
\begin{lemma}\label{lemma:finalprob}
Let $r_m, r_n \in \cR$ be two requests whose arrival rate follow a Poisson process with parameters $\alpha_m$ and $\alpha_n$, respectively. The probability of each having an occurrence within a maximum time interval $\bar{t}$ is
\begin{equation}\label{eq:lem}
	P(\alpha_m,\alpha_n)  := 1-\frac{\alpha_m e^{-\alpha_n\bar{t}}+\alpha_n e^{-\alpha_m\bar{t}}}{\alpha_m+\alpha_n}.
\end{equation}
\end{lemma}
\begin{proof}
	The proof can be found in Appendix~\ref{app:1}.
\end{proof}

\subsubsection{Expected Number of Pooled Rides}
\label{sec:STF}
In Section~\ref{sec:SD}, we analyzed the spatial dimension of the ride-pooling problem, whereby we computed the best feasible pooling path given two requests. In Section~\ref{sec:TD}, we analyzed the temporal dimension of the ride-polling problem, whereby we derived the probability of two requests happening within a time window.  By lifting the temporary assumptions made in Section~\ref{sec:SD}, we formulate the ride-pooling demand matrix given a certain pooling assignment, defined in what follows.

A fraction of the demand of every request $r_m\in \cR$ can be assigned to be pooled with a request $r_n \in \cR$. Let ${\beta \in \mathbb{R}_{\geq0}^{|\cR|\times |\cR|}}$ denote the assignment matrix, whose entry $(m,n)$ is the demand of $r_m$ that is assigned to be pooled with $r_n$. 
For the remainder of this subsection we assume that $\beta$ is given. In Section~\ref{sec:alg}, we propose an algorithm to compute the optimal value of $\beta$ under Approximation~\ref{approx}.

From the analysis in Section~\ref{sec:TD}, it is noticeable that only a fraction of the allocated ride-pooling demand $\beta_{mn}$ can actually be pooled due to the aforementioned temporal constraints. Specifically, the probability of pooling is given by $P(\beta_{mn},\beta_{nm})$ according to Lemma~\ref{lemma:finalprob}. Moreover, given that we only consider pooling between two requests, at most, the maximum pooled demand between $r_m,r_n\in \cR$ is  $\min(\beta_{mn},\beta_{nm})$. Therefore, the effective expected pooled demand between two requests $r_m,r_n\in \cR$ is given by ${\gamma_{nm} = \gamma_{mn} := \min(\beta_{mn},\beta_{nm})P(\beta_{mn},\beta_{nm})}$.  As a result, according to the spatial analysis in Section~\ref{sec:SD}, this pooled demand is portrayed by the demand matrix $\gamma_{mn}D^{mn,\star}$. Note that the effective expected pooling demand follows $\sum_{n\in \cM} \gamma_{mn} \leq \alpha_m, \, \forall r_m \in \mathcal{R}$ with equality if the full demand of $r_m$ is pooled.
The full ride-pooling demand matrix $D^\mathrm{rp}$ is made up of two contributions: i)~the sum of the expected pooled active vehicle flows of the form $\gamma_{mn}D^{mn,\star}$ for ${r_m,r_n\in \cR}$; and ii)~the requested demands that were not ride-pooled. Thus, the entry $(i,j)$ of $D^\mathrm{rp}$ can be written as
\begin{equation*}\label{eq:b}
		\!D_{ij}^{\mathrm{rp}} = \begin{cases}
			\sum\limits_{\substack{p,q\in \cM\\ p\geq q}} \gamma_{pq}D_{ij}^{pq,\star} + \left(D_{ij} - \sum\limits_{p\in\cM}\gamma_{mp}D_{ij}^{mp,\star} \right), \\ 
			 \quad \quad  \quad \quad  \quad \quad  \quad \quad \;\; \exists m \in \mathcal{M} : d_m \!= i \land o_m\! = j\\
			- \sum_{k\neq j} D_{kj}^{\mathrm{rp}},  \quad \quad \quad  i  = j \\ 
				\sum\limits_{\substack{p,q\in \cM\\ p\geq q}} \gamma_{pq}D_{ij}^{pq,\star}, \quad \quad   \text{otherwise}.
		\end{cases}
\end{equation*}

Finally, one can input $D^\mathrm{rp}$ to Problem~\ref{prob:rides}, which yields an LP, given a pooling assignment $\beta$.

\subsubsection{Optimal Ride-pooling Assignment}\label{sec:alg}
In this section, we will compute the optimal ride-pooling assignment matrices $\beta^\star$ and $\gamma^\star$, under Approximation~\ref{approx}, leveraging an iterative approach, which is described in what follows.
For every pair of requests $r_m,r_n\in \cR$, we can compute the unitary improvement of the objective function of Problem~\ref{prob:rides}, denoted by $\Delta \tilde{J}_{mn}$, w.r.t. the no-pooling scenario. Specifically, it amounts to the difference  between $\tilde{J}_{nm}$, which denotes the cost with $D^\mathrm{rp} = D^{mn,\star}$, and $\tilde{J}_n+\tilde{J}_m$, which again denotes the cost with $D^\mathrm{rp} = D^{mn,0}$. Let $\alpha_m^\prime, m\in \cM$ stand for an auxiliary variable throughout the iterations and represent the demand of request $r_m$ that has not yet been assigned, and which is initialized as $\alpha_m^\prime = \alpha_m$. Further, the pair of requests with the highest improvement is prioritized with the highest possible pooling demand assignment. That is, in each iteration, if $r_m,r_n\in \cR$  is the pair of requests with the highest $\Delta \tilde{J}_{mn}$, we set $\beta_{mn} = \alpha_m^\prime$ and $\beta_{nm}=\alpha_n^\prime$. Moreover, the rides that have been assigned but not pooled, are added back to the original requests, i.e., we set $\alpha_m^\prime =  \beta_{mn}-\gamma_{mn}$ and $\alpha_n^\prime =\beta_{nm}-\gamma_{nm}$. Let $\Delta \tilde{J}_{mn}^\prime, m,n\in \cM$ denote another auxiliary variable  throughout the iterations, initialized as $\Delta \tilde{J}_{mn}^\prime=  \Delta \tilde{J}_{mn}$. At the end of every iteration, $\Delta \tilde{J}^\prime_{mn}$ is set to  0.   
This procedure is repeated until convergence is achieved, i.e., $\max_{m,n} (\Delta \tilde{J}^\prime_{mn}) \leq 0$. The pseudocode of this procedure is presented in Algorithm~\ref{alg:one}. In the following theorem, we establish the convergence and optimality of Algorithm~\ref{alg:one}.
	
\begin{algorithm}[ht]
	\caption{Compute optimal assignment matrices $\beta^\star, \gamma^\star$.}\label{alg:one}
	\begin{algorithmic}
		\STATE $\tilde{J}_{mn} \leftarrow  \mathrm{input} \;\; D^{mn, \star} \; \text{to Problem}~\ref{prob:rides}, \;\; \forall {m,n\in \cM}$
		\STATE $\tilde{J}_{m} + \tilde{J}_n  \leftarrow  \mathrm{input} \;\; D^{mn, 0} \; \text{to Problem}~\ref{prob:rides}, \;\; \forall {{m,n\in \cM}}$
		\STATE $\Delta \tilde{J}_{mn} \;\leftarrow\; \tilde{J}_{m} + \tilde{J}_{n} - \tilde{J}_{mn} $
		\STATE $\Delta \tilde{J}_{mn}^\prime \;\leftarrow\; \Delta \tilde{J}_{mn}, \; \forall m,n\in \cM$
		\STATE $\alpha_{m}^\prime \;\leftarrow\; \alpha_{m}, \; \forall m\in \cM $
		\WHILE {$ \mathrm{max}_{m,n}(\Delta \tilde{J}^\prime_{mn}) > 0 $}
		\STATE $ (m,n) \in \mathrm{argmax}_{m,n} (\Delta \tilde{J}^\prime_{mn})$
		\IF{$o_n = o_m \; \mathrm{and} \; d_n=d_m$}
		\STATE $ \beta_{mn}  \;\leftarrow\; \alpha_{m}^\prime, \;\beta_{nm} \;\leftarrow\; \beta_{mn}$
		\STATE $ \gamma_{mn} \;\leftarrow\; \beta_{mn}  P(\beta_{mn},\beta_{nm})/2, \;\gamma_{nm} \;\leftarrow\; \gamma_{mn}$
		\ELSE
		\STATE $\beta_{nm}  \;\leftarrow\;  \alpha_{n}^\prime, \; \beta_{mn}  \;\leftarrow\;  \alpha_{m}^\prime$
		\STATE $\gamma_{mn}  \;\leftarrow\;  \min(\beta_{nm},\beta_{mn}) P(\beta_{mn},\beta_{nm})$
		\STATE $\gamma_{nm} \;\leftarrow\;  \gamma_{mn}$
		\ENDIF
		\STATE$\alpha_{m}^\prime  \;\leftarrow\;  \alpha_{m}^\prime - \gamma_{mn},  \; \alpha_{n}^\prime  \;\leftarrow\; \alpha_{n}^\prime - \gamma_{nm} $
		\STATE$\Delta \tilde{J}^\prime_{mn}  \;\leftarrow\; 0, \;   \Delta \tilde{J}^\prime_{nm} \leftarrow \Delta \tilde{J}^\prime _{mn}  $
		\ENDWHILE
	\end{algorithmic}
\end{algorithm}

\begin{theorem}\label{theorem:one}
Let $X^\star_\gamma$ denote the optimal solution of Problem~\ref{prob:rides}, under Approximation~\ref{approx}, for the effective ride-pooling demand matrix $\gamma$. Then, in $\abs{\cM}(\abs{\cM}-1)$ iterations at most, Algorithm~\ref{alg:one} converges to $\beta = \beta^\star$ and $\gamma= \gamma^\star$, which is a minimizer of $\tilde{J}(X^\star_\gamma)$ among all valid effective ride-pooling matrices.
\end{theorem} 
\begin{proof}
		The proof can be found in Appendix~\ref{app:2}. 
\end{proof}

\subsection{Discussion}
A few comments are in order.
The mobility system is analyzed at steady-state, which is unsuitable for an online implementation, but it is appropriate for planning and design~\cite{LukeSalazarEtAl2021,SalazarLanzettiEtAl2019}. 
Then, Problems~\ref{prob:main} and \ref{prob:rides} allow for fractional flows, which is acceptable because of the mesoscopic perspective of the work~\cite{PaparellaChauhanEtAl2023,LukeSalazarEtAl2021,SalazarLanzettiEtAl2019}.
We consider the travel time of each arc to be constant, meaning that the routing strategies do not impact travel time and congestion. Finally, $D^\mathrm{rp}$ is not optimal w.r.t. the objective function of Problem~\ref{prob:rides}, but it is w.r.t. its relaxed version, enabling a polynomial-time computation.
\section{Case Study}\label{sec:res}
This section showcases our modeling and optimization framework in a real-world case study of Sioux Falls, South Dakota, USA, with data obtained from the Transportation Networks for Research repository~\cite{ResearchCoreTeam}. The road network is shown in Fig.~\ref{fig:road}.
\begin{figure}[t]
	\centering
	\includegraphics[trim={0cm 300 0 80},clip,width=\linewidth]{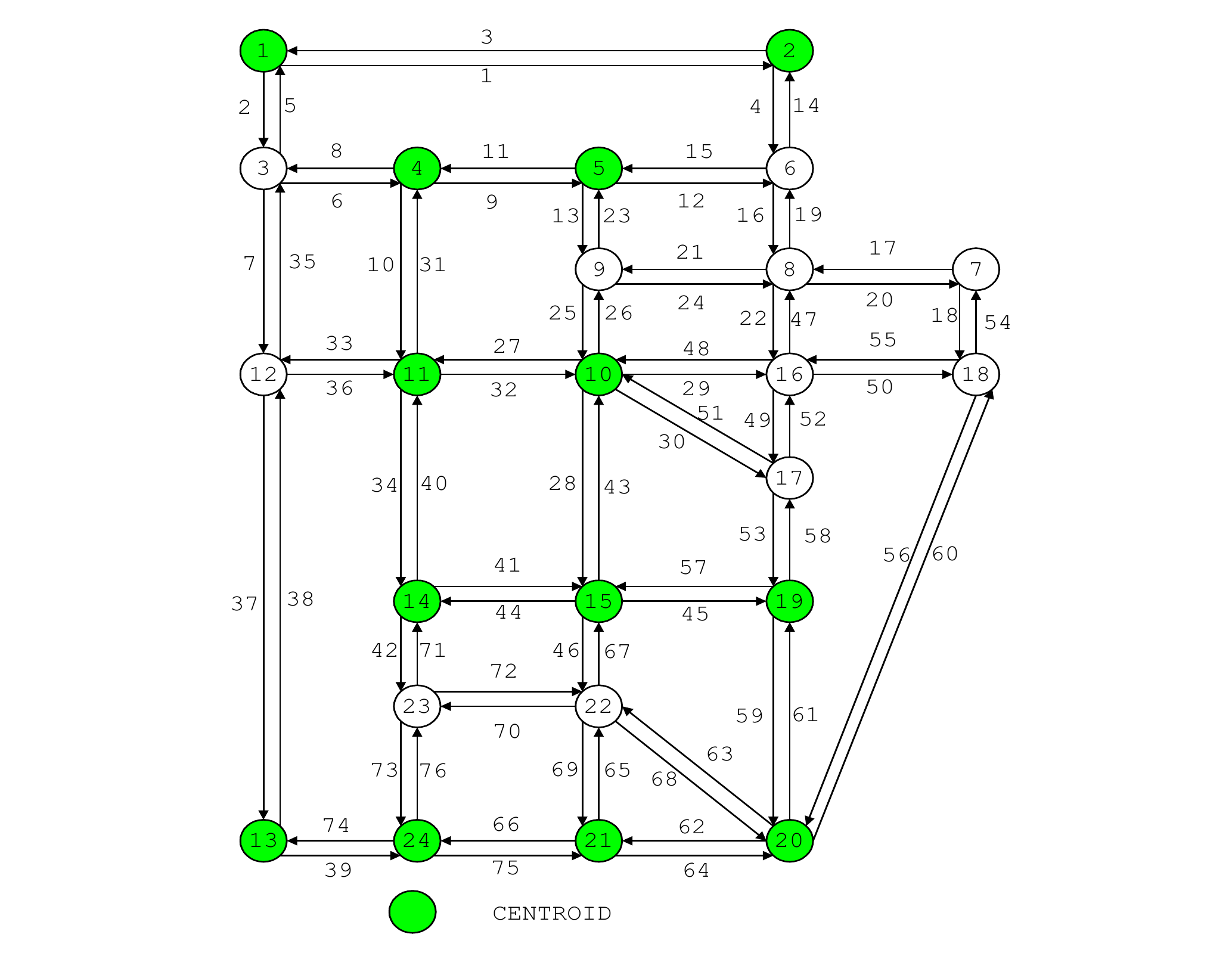}
		\caption{Road network of Sioux Falls, South Dakota, USA. Prepared by Hai Yang and Meng Qiang, Hong Kong University of Science and Technology.}
	\label{fig:road}
\end{figure}
The computations were performed on an Intel core i7-10850H, 32GB RAM. 
Problems~\ref{prob:rides} was parsed with YALMIP~\cite{Loefberg2004} and solved with Gurobi 9.5~\cite{GurobiOptimization2021}.
The computation of the matrices $D^{mn,\star}$ with $m,n \in \cM$ required less than 10 wall-clock minutes. Each instance of Problems~\ref{prob:rides} took less than 1 minute to solve. We compute it for a varying amount of hourly demands, waiting times and experienced delays considering the optimal ride-pooling assignment of the relaxed problem, obtained as described in Section~\ref{sec:alg}. Then, we compare the objectives of Problem~\ref{prob:main} and \ref{prob:rides}, i.e. the overall travel time.  Fig.~\ref{fig:WD} shows that ride-pooling always contributes to lowering the overall travel time. In particular, the larger the number of hourly demands, the larger the difference with respect to the no-pooling scenario. The reason is that the probability function in~\eqref{eq:lem} is monotonically increasing w.r.t. $\beta_m$ and $\beta_n$ that in turn, are monotonically increasing with the number of demands. 
\begin{figure}[t]
	\centering
	\includegraphics[width=\linewidth]{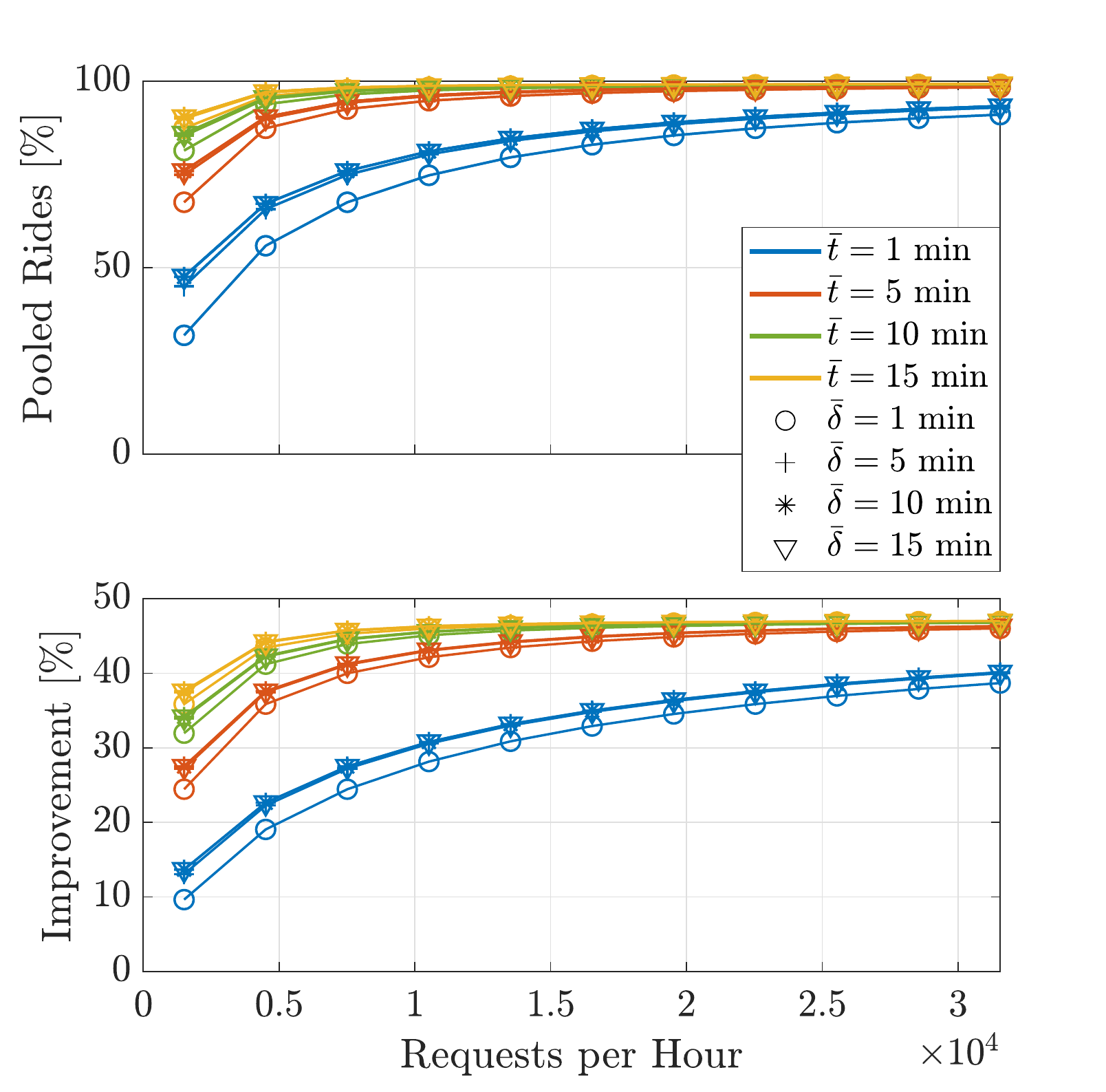}
	\caption{Percentage of pooled rides, objective function of Problem~\ref{prob:rides}, and improvement w.r.t. no pooling as a function of the overall number of hourly demands, waiting time, and experienced delay.}
	\label{fig:WD}
\end{figure}
\begin{figure}[t]
	\centering
	\includegraphics[width=\linewidth]{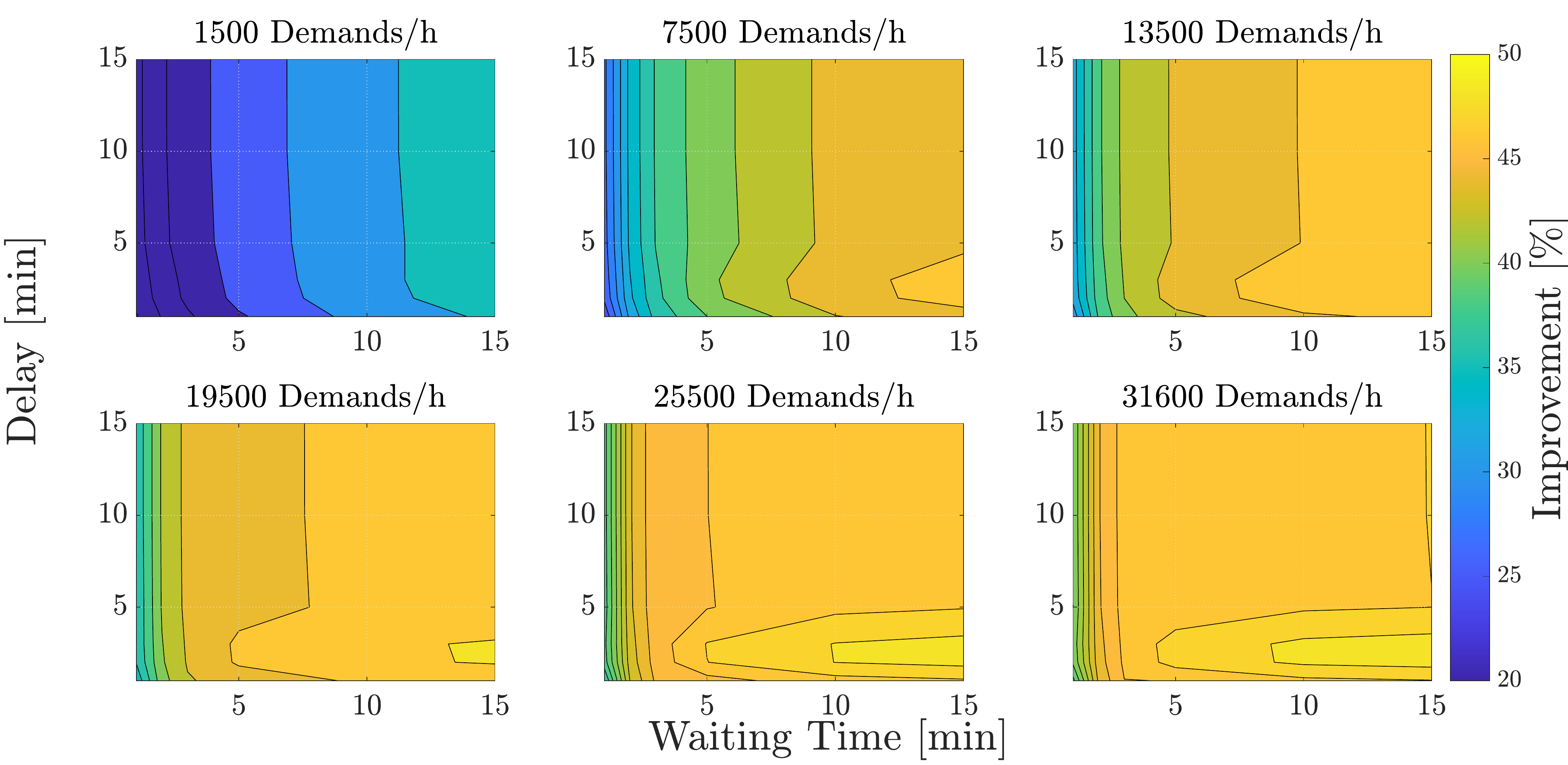}
	\caption{Improvement of $J$ with ride-pooling w.r.t. no ride-pooling,  as a function of maximum waiting time and delay.}
	\label{fig:surf}
\end{figure}%
We also note that the percentage of rides that are pooled is strongly influenced by the number of demands, to a lower extent by the maximum waiting time, and marginally by the maximum delay. In fact, for large demands, both the waiting time and the delay have a minor impact on the percentage of rides being pooled and on the costs, as shown in Fig.~\ref{fig:WD}. This phenomenon resembles the Mohring Effect~\cite{FielbaumTirachiniEtAl2021}, stating that the more people use a mobility service, the lower the waiting time they experience. 
We can also find the same effect in Fig.~\ref{fig:surf}, which shows that the higher the number of demands per unit time, the lower the delay and waiting time to obtain the same improvement with respect to the no-pooling scenario. Moreover, it is usually more beneficial to increase the waiting time of users rather than increase the delay, so that less distance is driven by the fleet, reducing the costs.

\section{Conclusions}\label{sec:conc}
This paper presented a framework to capture ride-pooling in a time-invariant network flow model. Specifically, we proposed a framework wherein we devise an equivalent set of requests w.r.t. to the original set so that the structure of the traffic flow  problem remains unchanged. This allows to still obtain an LP problem that can be efficiently solved with off-the-shelf solvers in polynomial-time. Additionally, we proposed a method to compute a ride-pooling request assignment, that is optimal w.r.t. a relaxed version of the minimum travel time problem. Our case study of Sioux Falls, USA, quantitatively showed that the overall number of requests per unit time is a crucial factor to assess the benefit of ride-pooling in mobility-on-demand system. In fact, we achieved average improvements from 25\% to 45\% for an increasing number of requests.  We also showed that, for a large number of requests, more than 90\% of them could be pooled with a relatively short waiting and delay time.



In the future, we would like to analyze the results with respect to the granularity of the road graph. Moreover, we would like to build on this research by including endogenous traffic congestion and by applying this method to other problems that can be approached with linear time-invariant traffic flow models. 


\section*{Statement of Code Availability}
A MATLAB implementation of the methods presented is available in an open-source repository at {\small \url{https://github.com/fabiopaparella/ride-pooling-MoD}}.

\section*{Acknowledgments}\label{Sec:akn}
We thank Dr. I. New, F. Vehlhaber, and J. Kampen for proofreading the paper. This publication is part of the
project NEON with number 17628 of the research
program Crossover, partly financed by the Dutch
Research Council.

\bibliographystyle{IEEEtran}
\bibliography{main.bib,SML_papers.bib}
\appendices
\section{Proof of Lemma~\ref{lemma:finalprob}}\label{app:1}


Recall that the exponential distribution, whose probability density function is given by $f(x)= \alpha e^{-\alpha x}$, models the time between events in a Poisson process of parameter $\alpha$. Since the two Poisson processes are independent,
\begin{equation*}
	\begin{split}
		P(\alpha_m,\alpha_n) = & \int_0^{\bar{t}} \alpha_n  e^{-\alpha_n  t_n} \left( \int_0^{t_n+\bar{t}} \alpha_m  e^{-\alpha_m t_m} \mathrm{d}t_m \right) \mathrm{d}t_n + \\
		& \int_{\bar{t}}^{\infty} \alpha_n e^{-\alpha_n t_n} \left(\int_{t_n-\bar{t}}^{t_n+\bar{t}} \alpha_m e^{-\alpha_m t_m}  \mathrm{d}t_m \right)  \mathrm{d}t_n,
	\end{split}
\end{equation*}
where the presence of two terms arises from the fact that the time interval $[0,+\infty)$ is considered. Making use of standard integral calculus techniques, it can be rewritten as \eqref{eq:lem}.

\section{Proof of Theorem~\ref{theorem:one}}\label{app:2}

The convergence of Algorithm~\ref{alg:one} in at most  ${\!\abs{\cM}(\abs{\cM}\!-\!1)\!}$ iterations is immediate. In fact, since for each pair $(m,n)$ chosen in each iteration we set $\Delta \tilde{J}^\prime_{mn} \!= \!\Delta \tilde{J}^\prime_{nm} \!= \!0$, neither $(n,m)$ nor $(m,n)$ will be chosen again. The optimality of the solution $\beta^\star$ and associated $\gamma^\star$ is carried out making use of an analogy with the continuous Knapsack problem,  which can be solved by a well-known polynomial-time greedy algorithm~\cite{Dantzig1957}. Recall that such algorithm consists in, every iteration, allocating the maximum amount of the resource with the highest improvement in the objective function per unit of  the resource, which is intuitively evident.  Similarly to the continuous Knapsack problem, the goal is to minimize $\tilde{J}(X^\star_\gamma)$ by allocating ${\gamma_{mn} \geq 0}$ with $m,n \in \cM$. First, borrowing the notation from Section~\ref{sec:SD}, if $\gamma_{mn}$ is assigned, then the corresponding decrease in the cost function amounts to $\tilde{J}(\gamma_{nm}X^{mn,0})\! -\!\tilde{J}(\gamma_{mn}X^{mn,\star}) =\gamma_{mn}(\tilde{J}(X^{mn,0})\! -\!\tilde{J}(X^{mn,\star})) = \gamma_{mn}\Delta{\tilde{J}}_{mn}$, where the linearity of $\tilde{J}$ played a key role. Thus, the allocation of $\gamma_{mn}$ leads to a relative improvement on the cost that amounts to $\Delta{\tilde{J}}_{mn}$. Second, as pointed out in Section~\ref{sec:alg}, throughout the algorithm, $\alpha^\prime_m$ corresponds to the demand of $r_m$ which has not yet been ride-pooled with another request. Thus, the value of $\gamma_{mn}$ that can be allocated has an upper bound given by $\gamma_{mn} \leq \min(\alpha_m^\prime, \alpha_n^\prime)P(\alpha_m^\prime,\alpha_n^\prime)$. Note that Algorithm~\ref{alg:one} corresponds to allocating the maximum amount of $\gamma_{mn}$, where $m$ and $n$ are such that, at each iteration, the highest positive relative improvement in the objective function is achieved, i.e., $(m,n) \in \mathrm{argmax}_{m,n} (\Delta \tilde{J}^\prime_{mn})$, which shows its optimality.

\end{document}